\documentclass[10pt]{article}
\usepackage{tocloft}

\usepackage{latexsym,amsfonts,euscript,amsthm}

\usepackage{amssymb}
\usepackage{stmaryrd}
\usepackage{mathrsfs,amsmath}
\usepackage{leftidx}
\usepackage{mathabx}

\usepackage[shortlabels]{enumitem}

\usepackage[figuresright]{rotating}
\usepackage[table]{xcolor}
\usepackage{multirow}
\usepackage{booktabs}
\usepackage{tabularx}

\usepackage{tocbibind} 

\usepackage{float}
\usepackage{tikz-cd}
\usetikzlibrary{cd}

\usepackage[pagebackref]{hyperref}

\hypersetup{%
  colorlinks=true,
  linktoc=all,
  citebordercolor=green,
  linkbordercolor=red,
 citecolor=black,
    linkcolor=blue,
  urlcolor=blue
}

\makeatletter
\renewcommand{\@maketitle}{%
  \newpage
  \null
  \vskip 2em%
  \begin{center}%
    \let \footnote \thanks
    {\LARGE \bfseries \@title \par}
    \vskip 1.5em%
    {\large \textsc{\@author}\par}%
    \vskip 1em%
    {\large \@date \par}%
  \end{center}%
  \par
  \vskip 1.5em
}
\makeatother

\renewcommand{\contentsname}{\small\bfseries Contents}

\newcommand{\centeredtitle}[1]{%
  \begin{center}
    {\small\bfseries #1}%
  \end{center}
  \vspace{1ex}%
}

\makeatletter
\AtBeginDocument{%
  \renewcommand{\tableofcontents}{%
    \centeredtitle{\contentsname}%
    \@starttoc{toc}%
  }%
}
\makeatother

\renewenvironment{abstract}{%
  \centeredtitle{\abstractname}%
  \vspace{-0.6\baselineskip}
  \quotation
}{%
  \endquotation
}

\newcommand{\ackname}{Acknowledgements}
\makeatletter
\if@titlepage
  \newenvironment{acknowledgement}{%
    \titlepage
    \null\vfil
    \@beginparpenalty\@lowpenalty
    \begin{center}%
      \bfseries \ackname
      \@endparpenalty\@M
    \end{center}}%
  {\par\vfil\null\endtitlepage}
\else
  \newenvironment{acknowledgement}{%
    \if@twocolumn
      \section*{\ackname}%
    \else
      \small
      \begin{center}%
        {\bfseries \ackname\vspace{-0.5em}\vspace{\z@}}%
      \end{center}%
      \quotation
    \fi}
    {\if@twocolumn\else\endquotation\fi}
\fi
\makeatother

\makeatletter 
\@addtoreset{equation}{section}
\makeatother  

\makeatletter
\def\thanks#1{\g@addto@macro\@thanks{\footnotetext{#1}}}
\makeatother

\makeatletter
\newtoks\address@list
\newcommand{\addaddress}[3]{%
  \address@list=\expandafter{\the\address@list
    \textsc{#1, #2.} \par
    \textit{Email address}: \href{mailto:#3}{\textsf{#3}} \par
    \addvspace{\medskipamount}%
  }%
}
\newcommand{\printaddresses}{%
  \AtEndDocument{\bigskip{\footnotesize
    \the\address@list
  }}%
}
\makeatother

\newtheorem{lem}{\bf Lemma}[section]

\newtheorem{mainthm}{Theorem}

\newcommand{\PSL}{{\operatorname{PSL}}}

\title{Hall complement numbers
\thanks{
\textbf{2020 MR Subject Classification}\,\, Primary 20D20\\
The second author is supported by Natural Science Foundation of Shanghai
(24ZR1422800) and National Natural Science Foundation of China (12471018).
}}

\author{\textsc{Yu Zeng}, \textsc{Hangyang Meng}
*\thanks{*Corresponding author.}\\
}
\addaddress{Yu Zeng}{Department of Mathematics, Suzhou University of Technology, No.99 South Third Ring Road, Changshu, Jiangsu, 215500, China}{yuzeng2004@163.com}
\addaddress{Hangyang Meng}{Department of Mathematics
	and Newtouch Center for Mathematics of Shanghai University, Shanghai University, Shanghai, 200444, China}{hymeng2009@shu.edu.cn}
\printaddresses

\date{}

\begin{document}

\maketitle

\begin{abstract}
A positive integer $m$ is termed a \emph{Hall number} if every finite group $G$ whose order is precisely divisible by $m$ possesses a Hall subgroup of order $m$.
Seeking generalizations of Sylow's theorem and Hall's theorem for finite solvable groups, Jiping Zhang asked for a full classification of Hall numbers, a problem recently solved by Guo, Hu and Li. 
Inspired by Zhang's problem, Guohua Qian put forward an analogous problem on a full
classification of 
Hall complement numbers.
Recall that a positive integer $m$ is called a \emph{Hall complement number} provided that every finite group $G$ 
with $m$ precisely dividing $|G|$ admits a Hall subgroup of order $|G|/m$. In the present paper, we prove that every Hall complement number is either $1$ or of the form $4k+2$ for some non-negative integer $k$, thus answering Qian's problem.
\end{abstract}

\section{Introduction}

A positive divisor $m$ of a positive integer $n$ is called a \emph{Hall divisor} if $n$ is precisely divisible by $m$, i.e. $n/m$ is an integer such that $\gcd(m, n/m) = 1$.

A positive integer $m$ is termed a \emph{Hall number} if every finite group $G$ for which $m$ is a Hall divisor of $|G|$ possesses a Hall subgroup of order $m$. 
With the aim of extending Sylow's theorem and Hall's theorem for finite solvable groups, Jiping Zhang raised the problem of classifying all Hall numbers. 
By Sylow's theorem, every prime power is a Hall number. 
In addition, Cayley's normal $p$-complement theorem and Hall's theorem for finite solvable groups imply that any integer of the form $4k+2$, for some non-negative integer $k$, is a Hall number. 
Recently, Guo, Hu, and Li completed the full classification by proving that the remaining Hall numbers are exactly $12$, $24$, and $60$, thus completely resolving Zhang's problem.

Inspired by Zhang's problem, Guohua Qian propose a closely related problem: the full classification of Hall complement numbers. 
Throughout this paper, a positive integer $m$ is called a \emph{Hall complement number} if every finite group $G$ for which $m$ is a Hall divisor of $|G|$ admits a Hall subgroup of order $|G|/m$. 
In our main theorem, we give a complete solution to Qian's problem.

\begin{mainthm}\label{thmA}
  A positive integer $m$ is a Hall complement number if and only if 
  either $m=1$ or $m=4k+2$ for some non-negative integer $k$.
\end{mainthm}

Throughout the rest of this paper, we only consider finite groups.

\section{Proofs}

We follow \cite{huppert67} for standard notation and terminology in finite group theory. Throughout the paper, $p$ always denotes a prime. 
Other notation will be recalled or defined when necessary.

We begin by verifying that every positive integer of the form $4k+2$ is a Hall complement number.

\begin{lem}\label{lem: 4k+2}
 For every non-negative integer $k$,
  $m=4k+2$ is a Hall complement number.
\end{lem}
\begin{proof}
Let $G$ be a finite group such that $m$ is a Hall divisor of $|G|$.
Since $m = 4k+2 = 2(2k+1)$, the exact power of $2$ dividing $m$ is $2$.
Because $m$ is a Hall divisor of $|G|$, the exact power of $2$ dividing $|G|$ is also $2$.
Thus, $G$ has a Sylow $2$-subgroup of order $2$.
By Cayley's normal $p$-complement theorem (see \cite[IV, 2.8 Satz]{huppert67}), $G$ has a normal $2$-complement $K$ of index $2$.
Since $K$ has odd order, $K$ is a solvable group by the Feit--Thompson theorem.
Because $G/K$ and $K$ are both solvable, $G$ itself is a solvable group.
Since $m$ is a Hall divisor of $|G|$, the integer $|G|/m$ is coprime to $m$.
By Hall's theorem, a solvable group contains a Hall subgroup of any permissible valid order.
In particular, $G$ must have a Hall subgroup of order $|G|/m$.
\end{proof}

The following well-known result is Dirichlet's theorem on arithmetic progressions, 
see, for instance, \cite[II, Theorem 15]{hardy08}.

\begin{lem}[Dirichlet, 1837]\label{lem-Dirichlet}
Let $a$ and $d$ be coprime integers, i.e., $\gcd(a,d)=1$. Then there exist infinitely many prime numbers $p$ such that $p \equiv a \pmod{d}.$
\end{lem}

The next key lemma is from \cite[Lemma 4]{guo25} due to Guo, Hu and Li.  
Here we will give an alternative proof using Dirichlet's theorem on arithmetic progressions. 

\begin{lem}\label{lem:Hall-divisor}
Let $m \ge 2$ be a positive integer and let $\xi$ be an integer such that $\gcd(\xi, m) = 1$. Then there exist infinitely many primes $p$ such that $p \equiv \xi \pmod{m}$ and
	\[
	\gcd\left(m, \frac{p-\xi}{m}\right) = 1.
	\]
\end{lem}
\begin{proof}
Let $m=p_1^{a_1}p_2^{a_2}\cdots p_t^{a_t}$, where $p_1< p_2<\cdots <p_t$ are primes, $t\geq 1$, and $a_i>0$.
Set $M=\prod_{i=1}^{t} p_i^{a_i+1}$.
By the Chinese Remainder Theorem, since the moduli $p_i^{a_i+1}$ are pairwise coprime for each $1 \leq i \leq t$, there exists a positive integer $A$ such that 
$$A \equiv \xi+p_i^{a_i} \pmod{p_i^{a_i+1}}$$
for $1\leq i\leq t$. Note that for each $i$, $A \equiv \xi+p_i^{a_i} \equiv \xi \pmod{p_i}$.
Since $\gcd(\xi,m)=1$, it implies that $\gcd(A, p_i) = 1$ for each $i$, and consequently, $\gcd(A, M)=1$.
By Lemma~\ref{lem-Dirichlet}, there exist infinitely many primes $p$ such that $p\equiv A\pmod M$. For such primes $p$, we have
$p \equiv A \equiv \xi+p_i^{a_i} \equiv \xi \pmod{p_i^{a_i}}$ for each $i$. Hence $m=p_1^{a_1}p_2^{a_2}\cdots p_t^{a_t}$ divides $p-\xi$, that is, $p \equiv \xi \pmod{m}$.

We next show that $\gcd(m,(p-\xi)/m)=1$. If $p_i$ divides $(p-\xi)/m$ for some $i$, then $p_im$ divides $p-\xi$, which implies that $p_i^{a_i+1}$ divides $p-\xi$. But $p-\xi \equiv A-\xi \equiv p_i^{a_i} \nequiv 0 \pmod{p_i^{a_i+1}}$, which is a contradiction. Hence $m$ is coprime to $(p-\xi)/m$, as desired.
\end{proof}

With the preceding number-theoretic lemma at our disposal, we now prove an auxiliary result that underpins our subsequent arguments involving the projective special linear group $\operatorname{PSL}_2(p)$. More precisely, we show that for any integer $m \ge 2$ of certain type, there exists a prime $p > 11$ such that $m$ is a Hall divisor of $|\mathrm{PSL}_2(p)|$.

\begin{lem}\label{lem:m odd or 4k}
Let $m \geq 2$ be a positive integer.
Then there exist an odd prime $p>11$ such that $m$ is a Hall divisor of $p-1$. Furthermore if $m$ is odd or $4$ divides $m$, then $m$ is also a Hall divisor of $ p(p-1)(p+1)/2$.
\end{lem}
\begin{proof}
Applying Lemma~\ref{lem:Hall-divisor} for $\xi=1$, there exists an odd prime  $p>11$ such that $m$ is a Hall divisor of $p-1$. Therefore we only need to prove the latter part. Note that
$\gcd(p-1, p(p+1)/2)$ divides 
$$\gcd(p-1,p(p+1))=\gcd(p-1,p+1)=\gcd(p-1,2)=2.$$
Hence $\gcd(p-1,p(p+1)/2)=1~\text{or}~2$. If $m$ is odd, as $m$ divides $p-1$, then
$\gcd(m, p(p+1)/2)=1$. Hence $m$ is a Hall divisor of $p(p-1)(p+1)/2$ as $m$ is a Hall divisor of $p-1$. 

Now we assume $4$ divides $m$. In this case, we have $4 \mid p-1$ and set $p-1=4k$ for some positive integer $k$. Then $p+1=4k+2=2(2k+1)$ and $(p+1)/2=2k+1$ is odd. Hence $\gcd(p-1, p(p+1)/2)=1$ and so $\gcd(m,p(p+1)/2)=1$. Hence $m$ is a Hall divisor of $p(p-1)(p+1)/2$, as desired. 
\end{proof}

Now, we are ready to prove Theorem \ref{thmA}.

\begin{proof}[\emph{Proof of Theorem \ref{thmA}}]
Clearly $1$ is a Hall complement number and
$m=4k+2$ for a non-negative 
integer $k$ is also a Hall complement number by Lemma \ref{lem: 4k+2}. Hence we only prove the necessity of this result.

We assume that $m$ is a Hall complement number such that $m$ is neither 1 nor $4k+2$ for every non-negative integer $k$. Hence we obtain that
$m \geq 2$ such that $m$ is odd or $4$ divides $m$.
By Lemma~\ref{lem:m odd or 4k}, there exists a prime $p>11$ such that $m$ is a Hall divisor of $p-1$ and $m$ is also a Hall divisor of $p(p-1)(p+1)/2$.

Let $G=\PSL_2(p)$ and clearly $|G|=p(p-1)(p+1)/2$.
Since $m$ is a Hall divisor of $|G|$, by the definition of Hall complement number, $G$ has a subgroup of order $|G|/m=p(p+1)(p-1)/m$, say $H$. As $m \geq 2$ and $p$ divides $|H|$, $H$ is a proper group of $G$ containing a Sylow $p$-subgroup of $G$. By Dickson's classification of the subgroups of $\mathrm{PSL}_2(p)$ (\cite[II, 8.27 Hauptsatz]{huppert67}), since $p > 11$, any proper subgroup of $G$ containing a Sylow $p$-subgroup must be contained in a Borel subgroup of order $p(p-1)/2$, which forces $|H|$ to divide $p(p-1)/2$.
Consequently, the ratio 
$$\frac{p(p-1)}{2|H|} = \frac{m}{p+1}$$ must be an integer, implying that $p+1$ divides $m$. This forces $p+1 \leq m$.
  However, since $m$ divides $p-1$, we must have $m \leq p-1$, leading to a clear contradiction $p+1 \leq p-1$.
\end{proof}

\begin{acknowledgement}
   The author is grateful to the referee for her/his valuable comments.
\end{acknowledgement}


\begin{thebibliography}{ABCDEF}\setlength{\itemsep}{-2mm} 
\small


\bibitem[GHL25]{guo25}
Z. Guo, Y. Hu and C.H. Li,
\newblock The exceptional Hall numbers,
\newblock \emph{Peking Math. J.} {\bf } (2025).  


   \bibitem[HW08]{hardy08}
   G. Hardy and E. Wright,
   \newblock \emph{An introduction to the theory of numbers}, sixth edition,
   \newblock Oxford University Press, Oxford, 2008.



   \bibitem[Hup67]{huppert67}
   B. Huppert,
   \newblock \emph{Endliche Gruppen I},
   \newblock Springer-Verlag, Berlin, 1967.



\end{thebibliography}
\end{document}